\documentclass[11pt,oneside,reqno]{amsart}

\usepackage{amsmath,amssymb,amsthm}
\usepackage[colorlinks=true,linkcolor=blue,citecolor=blue,urlcolor=blue]{hyperref}

\theoremstyle{plain}
\newtheorem{theorem}{Theorem}
\newtheorem{lemma}{Lemma}
\newtheorem{conjecture}{Conjecture}

\newtheorem{corollary}{Corollary}

\theoremstyle{remark}

\newcommand{\sone}{\sigma_1}
\newcommand{\coeff}[2]{[#1]#2}

\allowdisplaybreaks

\title[A conjecture of Amdeberhan-Andrews-Ballantine]{Two Proofs of a Conjecture of Amdeberhan, Andrews and Ballantine for double Lambert series and a new Representation for
$E_2(q)$}

\author{Su-Ping Cui}
\address{School of Mathematical Sciences, Qufu Normal University, Qufu 273165, P.R. China}
\email{jiayoucui@163.com}

\author{Rahul Kumar}
\address{Department of Mathematics, Indian Institute of Technology, Roorkee-247667, Uttarakhand, India}
\email{rahul.kumar@ma.iitr.ac.in}
\author{Aman Singh}
\address{Department of Mathematics, Indian Institute of Technology, Roorkee-247667, Uttarakhand, India}
\email{amansingh9839269454@gmail.com}

\begin{document}

\begin{abstract}
In this note, we prove a recent conjecture of Amdeberhan, Andrews and Ballantine concerning a double Lambert series  (\textit{J. Combin. Theory Series A} \textbf{221} (2026), Paper No. 106154). More precisely, they conjectured that
\[
\coeff{q^{N2^a}}
\sum_{m,k\geq 1}
\frac{q^{mk2^a}}{(1+q^{k2^{a-1}})(1-q^{2m-1})}
=\sigma_1(N),
\]
where $\sigma_1(N)$ is the sum of all the positive divisors of $N$. We provide two proofs of this conjecture. One of the approach leads us to derive a new representation of quasi-modular forms $E_2(q)$.
\end{abstract}

\maketitle

\section{Introduction}
A series of the form
\begin{align}\label{lambert}
\sum_{n\geq1}\frac{a_nq^n}{1-q^n} \quad (|q|<1)
\end{align}
is known as \emph{Lambert series}, named after J. H. Lambert \cite{lambert}, who showed in 1771 that the generating function for the divisor function
is given by
\begin{align*}
\sum_{n\geq1}\frac{q^n}{1-q^n}.
\end{align*}
Such series play an important role in number theory and more broadly throughout mathematics. For example, when $a_n=n^{2k-1}$, the resulting Lambert series are essentially modular forms for $k>1$ and quasi-modular forms for $k=1$.

Recently, Amdeberhan, Andrews and Ballantine \cite{AAB} undertook a systematic study of generalized Lambert series and double Lambert series of the types
\begin{align*}
\sum_{n\geq1}R_n\left(q^n,q\right)\qquad \mathrm{and}\qquad \sum_{m,n\geq1}S_{n,m}\left(q^n,q^m,q\right),
\end{align*}
where $R_n\left(x,z\right)$ and $S_{n,m}\left(x,y,z\right)$ are rational functions of $x,y$ and $z$. They established several connections between these series and Rogers--Ramanujan type $q$-series. In the course of their investigation, they proposed the following two conjectures \cite[p.~21, Conjecture 5.12 \& Conjecture 5.13]{AAB}:

\begin{conjecture}[\cite{AAB}]\label{aab conjecture 1}
Let $a$ be a positive integer. Then, for each positive integer $N$, we have
\begin{equation*}
\coeff{q^{N2^a}}
\sum_{m,k\geq 1}
\frac{q^{mk2^a}}{(1+q^{k2^{a-1}})(1-q^{2m-1})}
=\sone(N),
\end{equation*}
where $\displaystyle\sigma_1(N):=\sum_{0<d|N}d$. Here and throughout the sequel, $[q^N]A(q)$ denotes the coefficient of $q^N$ in the series $A(q)$.
\end{conjecture}

\begin{conjecture}[\cite{AAB}]\label{aab conjecture 2}
If $r$ is a positive integer, then
\begin{align*}
\coeff{q^{2r}}\sum_{m,n\geq1}\frac{q^{2mn}}{(1+q^{2n-1})(1-q^{2m-1})}=\coeff{q^{2r}}\sum_{n\geq1}\frac{(n-1)q^n}{1+q^{2n-1}}.
\end{align*}
\end{conjecture}

Interestingly, Amdeberhan et al. \cite{AAB} proposed these conjectures in the quest of proving the following conjecture of Andrews, Dixit, Schultz and Yee \cite[p.~24, Problem 2]{adsy}:
\begin{conjecture}[\cite{adsy}]\label{adsy conjecture}
The function
\begin{align*}
Y(q):=\sum_{m,n\geq1}\frac{(-q)^{2mn+m}}{(1+q^n)(1-q^{2m-1})}
\end{align*}
 is an odd function of $q$.
\end{conjecture}
Andrews et al. \cite{adsy} were led to this conjecture in their study of combinatorial interpretations of mock theta functions.

Very recently, Conjecture \ref{aab conjecture 2} and Conjecture \ref{adsy conjecture} have been proved by the first author and Tang \cite{ct}. In particular, Conjecture \ref{adsy conjecture} follows immediately from their elegant identity for $Y(q)$ \cite[Theorem 1.2]{ct}:
\begin{align}\label{ct}
Y(q)=-q\frac{(q^4;q^4)_\infty^4}{(q^2;q^2)_\infty^2}\sum_{k\geq1}\frac{q^{2k}}{1+q^{2k}},
\end{align}
where $(a;q)_\infty$ is the standard $q$-Pochhammer symbol:
\begin{align*}
(a;q)_n:&=\prod_{j=0}^{n-1} (1-aq^j),\\
(a;q)_\infty:&=\prod_{j=0}^{\infty} (1-aq^j).
\end{align*}

We here note that Fang \cite{fang} also independently  proved Conjecture \ref{adsy conjecture}.

In this note, we settle the remaining Conjecture \ref{aab conjecture 1}. That is, we prove the following result:
\begin{theorem}\label{thm:main}
Let $a$ and $N$ be positive integers.  Then
\begin{equation}\label{eq:main}
\coeff{q^{N2^a}}
\sum_{m,k\geq 1}
\frac{q^{mk2^a}}{(1+q^{k2^{a-1}})(1-q^{2m-1})}
=\sone(N).
\end{equation}
\end{theorem}

The main ingredient in one of our proofs of the above theorem is the following transformation which converts the double Lambert series into a single Lambert series.
\begin{theorem}\label{lem:base}
For $|q|<1$,  we have,
\begin{equation}\label{eq:base}
\sum_{m,k\geq 1}
\frac{q^{2mk}}{(1+q^k)(1-q^{2m-1})}
=
\sum_{n\geq 1}\frac{q^{2n}}{(1-q^{2n})^2}=\sum_{N\geq 1}\sone(N)q^{2N}.
\end{equation}
\end{theorem}

As mentioned before, the case $a_n=n$ of the series \eqref{lambert} is essentially the quasi-modular form $E_2(q)$ given by
\begin{align*}
E_2(q):=1-24\sum_{n\geq1}\sigma_1(n)q^{2n},
\end{align*}
where $q=e^{\pi i\tau},\ \tau\in\mathbb{H}:=\{z\in\mathbb{C}: \mathrm{Im}(z)>0\}.$

An interesting consequence of Theorem \ref{lem:base} is the following representation for $E_2(q)$, which appears to be new.
\begin{corollary}\label{new representation}
For $|q|<1$, the following relation holds\textup{:}
\begin{align*}
E_2(q)=1-24\sum_{m,k\geq 1}
\frac{q^{2mk}}{(1+q^k)(1-q^{2m-1})}.
\end{align*}
\end{corollary}

 \textit{Note Added.} We were informed by Professor Andrews and Professor Ballantine that Caner Nazaroglu also proved Conjecture \ref{aab conjecture 1} independently in an unpublished work.



\section{First Proof of Amdeberhan-Andrews-Ballantine Conjecture}\label{proof1}
We provide some lemmas which play important roles in deriving our first proof of Theorem \ref{thm:main}.
\begin{lemma}\label{lem-1} We have
\begin{align*}
\sum^{\infty}_{m=1}\frac{q^{2m-1}}{1-q^{2m-1}}\sum^{\infty}_{n=m}\frac{q^{2n}}{1-q^{2n}}
=\sum^{\infty}_{m=1}\frac{1}{1-q^{2m-1}}\sum^{\infty}_{n=m+1}\frac{q^{2n-1}}{1-q^{2n-1}}.
\end{align*}
\end{lemma}
\begin{proof} Let
\begin{align*}
A(q):=\sum^{\infty}_{m=1}\frac{q^{2m-1}}{1-q^{2m-1}}\sum^{\infty}_{n=m}\frac{q^{2n}}{1-q^{2n}}.
\end{align*}
Then
\begin{align}
A(q)&=\sum^{\infty}_{m=1}\frac{q^{2m-1}}{1-q^{2m-1}}\sum^{\infty}_{n=1}\frac{q^{2n}}{1-q^{2n}}
-\sum^{\infty}_{m=1}\frac{q^{2m-1}}{1-q^{2m-1}}\sum^{m-1}_{n=1}\frac{q^{2n}}{1-q^{2n}}\nonumber\\
&=\sum^{\infty}_{m=1}\frac{q^{2m-1}}{1-q^{2m-1}}\sum^{\infty}_{n=1}\frac{q^{2n}}{1-q^{2n}}
-\sum^{\infty}_{n=1}\frac{q^{2n}}{1-q^{2n}}\sum^{\infty}_{m=n+1}\frac{q^{2m-1}}{1-q^{2m-1}}.\label{1-1}
\end{align}
Next,
\begin{align}
\sum^{\infty}_{m=n+1}\frac{q^{2m-1}}{1-q^{2m-1}}
&=\sum^{\infty}_{k=1}\frac{q^{2n+2k-1}}{1-q^{2n+2k-1}}\nonumber\\
&=\sum^{\infty}_{k=1}q^{2n+2k-1}\sum^{\infty}_{m=0}q^{2mn+2km-m}\nonumber\\
&=\sum^{\infty}_{m=0}q^{2mn+2n-m-1}\sum^{\infty}_{k=1}q^{(2m+2)k}\nonumber\\
&=\sum^{\infty}_{m=0}\frac{q^{2mn+2n+m+1}}{1-q^{2m+2}}\nonumber\\
&=\sum^{\infty}_{m=1}\frac{q^{2mn+m}}{1-q^{2m}}. \label{1-2}
\end{align}
Substituting \eqref{1-2} into \eqref{1-1} yields
\begin{align}
A(q)=\sum^{\infty}_{m=1}\frac{q^{2m-1}}{1-q^{2m-1}}\sum^{\infty}_{n=1}\frac{q^{2n}}{1-q^{2n}}
-\sum^{\infty}_{m,n=1}\frac{q^{2mn+2n+m}}{(1-q^{2n})(1-q^{2m})}.\label{1-3}
\end{align}
We now turn to compute the second term in the above identity.
\begin{align}
&\sum^{\infty}_{m,n=1}\frac{q^{2mn+2n+m}}{(1-q^{2n})(1-q^{2m})}\nonumber\\
&=\sum^{\infty}_{m=1}\frac{q^{m}}{1-q^{2m}}\sum^{\infty}_{n=1}\frac{q^{2mn+2n}}{1-q^{2n}}\nonumber\\
&=\sum^{\infty}_{m=1}\frac{q^{m}}{1-q^{2m}}\sum^{\infty}_{n=1}q^{2mn+2n}\sum^{\infty}_{k=0}q^{2nk}\nonumber\\
&=\sum^{\infty}_{m=1}\frac{q^{m}}{1-q^{2m}}\sum^{\infty}_{k=0}\sum^{\infty}_{n=1}q^{(2m+2k+2)n}\nonumber\\
&=\sum^{\infty}_{m=1}\frac{q^{m}}{1-q^{2m}}\sum^{\infty}_{k=0}\frac{q^{2m+2k+2}}{1-q^{2m+2k+2}}\nonumber\\
&=\sum^{\infty}_{m=1}\frac{q^{m}}{1-q^{2m}}\sum^{\infty}_{n=m+1}\frac{q^{2n}}{1-q^{2n}}\nonumber\\
&=\sum^{\infty}_{n=1}\frac{q^{2n}}{1-q^{2n}}\sum^{n-1}_{m=1}\frac{q^{m}}{1-q^{2m}}\nonumber\\
&=\sum^{\infty}_{n=1}\frac{q^{2n}}{1-q^{2n}}\sum^{\infty}_{m=1}\frac{q^{m}}{1-q^{2m}}
-\sum^{\infty}_{n=1}\frac{q^{2n}}{1-q^{2n}}\sum^{\infty}_{m=n}\frac{q^{m}}{1-q^{2m}}.\label{1-4}
\end{align}

Then using steps similar to the proof of \eqref{1-2}, we find that
\begin{align}
&\quad\sum^{\infty}_{n=1}\frac{q^{2n}}{1-q^{2n}}\sum^{\infty}_{m=n}\frac{q^{m}}{1-q^{2m}}\nonumber\\
&=\sum^{\infty}_{m,n=1}\frac{q^{2mn+n}}{(1-q^{2n})(1-q^{2m-1})}\nonumber\\
&=\sum^{\infty}_{m=1}\frac{1}{1-q^{2m-1}}\sum^{\infty}_{n=1}\frac{q^{2mn+n}}{1-q^{2n}}\nonumber\\
&=\sum^{\infty}_{m=1}\frac{1}{1-q^{2m-1}}\sum^{\infty}_{n=m+1}\frac{q^{2n-1}}{1-q^{2n-1}}.\label{1-5}
\end{align}
Here the last step follows from \eqref{1-2}.
In terms of \eqref{1-4} and \eqref{1-5}, we find that
\begin{align}
\sum^{\infty}_{m,n=1}\frac{q^{2mn+2n+m}}{(1-q^{2n})(1-q^{2m})}
&=\sum^{\infty}_{n=1}\frac{q^{2n}}{1-q^{2n}}\sum^{\infty}_{m=1}\frac{q^{m}}{1-q^{2m}}\nonumber\\
&\quad-\sum^{\infty}_{m=1}\frac{1}{1-q^{2m-1}}\sum^{\infty}_{n=m+1}\frac{q^{2n-1}}{1-q^{2n-1}}.\label{1.7}
\end{align}
Consequently, from \eqref{1-3} and \eqref{1.7}, we conclude that
\begin{align*}
A(q)&=\sum^{\infty}_{m=1}\frac{q^{2m-1}}{1-q^{2m-1}}\sum^{\infty}_{n=1}\frac{q^{2n}}{1-q^{2n}}
-\sum^{\infty}_{n=1}\frac{q^{2n}}{1-q^{2n}}\sum^{\infty}_{m=1}\frac{q^m}{1-q^{2m}}\nonumber\\
&\quad+\sum^{\infty}_{m=1}\frac{1}{1-q^{2m-1}}\sum^{\infty}_{n=m+1}\frac{q^{2n-1}}{1-q^{2n-1}}.
\end{align*}
Thus, to prove the lemma, it suffices to show that
\begin{align*}
\sum^{\infty}_{m=1}\frac{q^{2m-1}}{1-q^{2m-1}}
=\sum^{\infty}_{m=1}\frac{q^m}{1-q^{2m}}.
\end{align*}
Notice that
\begin{align*}
\sum^{\infty}_{m=1}\frac{q^{2m-1}}{1-q^{2m-1}}&=
\sum^{\infty}_{m=1}\sum^{\infty}_{\ell=1}q^{(2m-1)\ell}
=\sum^{\infty}_{m=1}\sum^{\infty}_{\ell=1}q^{m(2\ell-1)}\nonumber\\
&=\sum^{\infty}_{m=1}\sum^{\infty}_{\ell=0}q^{m(2\ell+1)}
=\sum^{\infty}_{m=1}\sum^{\infty}_{\ell=0}q^{2m\ell+m}=\sum^{\infty}_{m=1}\frac{q^m}{1-q^{2m}}.
\end{align*}

Therefore, we complete the proof.
\end{proof}

\begin{lemma}\label{lem-2} For positive integer $r$ and $s$, we have
\begin{align}
\sum^{\infty}_{n=1}\frac{q^{(2n-1)s}}{1-q^{(2n-1)r}}
\sum^{\infty}_{m=n+1}\frac{q^{(2m-1)r}}{1-q^{(2m-1)r}}
=\sum^{\infty}_{m=1}\frac{q^{2rm}}{1-q^{2rm}}
\sum^{\infty}_{n=m}\frac{q^{rn+s}}{1-q^{2rn+2s}}.
\end{align}
\end{lemma}
\begin{proof} Let
\begin{align*}
B(r,s,q):=\sum^{\infty}_{n=1}\frac{q^{(2n-1)s}}{1-q^{(2n-1)r}}
\sum^{\infty}_{m=n+1}\frac{q^{(2m-1)r}}{1-q^{(2m-1)r}}.
\end{align*}
Then we find that
\begin{align}
B(r,s,q)
&=\sum^{\infty}_{n=1}\frac{q^{(2n-1)s}}{1-q^{(2n-1)r}}
\sum^{\infty}_{k=1}\frac{q^{(2n+2k-1)r}}{1-q^{(2n+2k-1)r}}\nonumber\\
&=\sum^{\infty}_{n=1}\frac{q^{(2n-1)s}}{1-q^{(2n-1)r}}
\sum^{\infty}_{k=1}q^{(2n+2k-1)r}\sum^{\infty}_{m=0}q^{(2mn+2km-m)r}\nonumber\\
&=\sum^{\infty}_{n=1}\frac{q^{(2n-1)s}}{1-q^{(2n-1)r}}
\sum^{\infty}_{m=0}q^{(2mn+2n-m-1)r}\sum^{\infty}_{k=1}q^{(2rm+2r)k}\nonumber\\
&=\sum^{\infty}_{n=1}\frac{q^{(2n-1)s+2rn+r}}{1-q^{(2n-1)r}}
\sum^{\infty}_{m=0}\frac{q^{2rmn+rm}}{1-q^{2rm+2r}}\nonumber\\
&=\sum^{\infty}_{n=1}\frac{q^{(2n-1)s}}{1-q^{(2n-1)r}}
\sum^{\infty}_{m=1}\frac{q^{2rmn+rm}}{1-q^{2rm}}.\label{2-1}
\end{align}
We now simplify \eqref{2-1}. Note that
\begin{align}
&\sum^{\infty}_{n=1}\frac{q^{(2n-1)s}}{1-q^{(2n-1)r}}
\sum^{\infty}_{m=1}\frac{q^{2rmn+rm}}{1-q^{2rm}}\nonumber\\
&=\sum^{\infty}_{m=1}\frac{q^{rm-s}}{1-q^{2rm}}
\sum^{\infty}_{n=1}\frac{q^{2ns+2rmn}}{1-q^{(2n-1)r}}\nonumber\\
&=\sum^{\infty}_{m=1}\frac{q^{rm-s}}{1-q^{2rm}}
\sum^{\infty}_{n=1}q^{2ns+2rmn}\sum^{\infty}_{k=0}q^{(2n-1)rk}\nonumber\\
&=\sum^{\infty}_{m=1}\frac{q^{rm-s}}{1-q^{2rm}}
\sum^{\infty}_{k=0}q^{-rk}\sum^{\infty}_{n=1}q^{(2rm+2rk+2s)n}\nonumber\\
&=\sum^{\infty}_{m=1}\frac{q^{rm-s}}{1-q^{2rm}}
\sum^{\infty}_{k=0}\frac{q^{2rm+rk+2s}}{1-q^{2rm+2rk+2s}}\nonumber\\
&=\sum^{\infty}_{m=1}\frac{q^{2rm}}{1-q^{2rm}}
\sum^{\infty}_{n=m}\frac{q^{rn+s}}{1-q^{2rn+2s}}. \label{2-2}
\end{align}
Combining \eqref{2-1} and \eqref{2-2} yields the desirable result.
This completes the proof.
\end{proof}

We are now ready to present the first proof of Conjecture \ref{aab conjecture 1}.

\noindent{\bf Proof of Theorem \ref{thm:main}.} Notice that
\begin{align}
&\sum^{\infty}_{m,n=1}\frac{q^{mn2^a}}{(1+q^{n2^{a-1}})(1-q^{2m-1})}\nonumber\\
&=\sum^{\infty}_{m,n=1}\frac{q^{mn2^a}(1-q^{n2^{a-1}})}{(1-q^{n 2^a})(1-q^{2m-1})}\nonumber\\
&=\sum^{\infty}_{m,n=1}\frac{q^{mn2^a}}{(1-q^{n 2^a})(1-q^{2m-1})}
-\sum^{\infty}_{m,n=1}\frac{q^{mn2^a+n2^{a-1}}}{(1-q^{n 2^a})(1-q^{2m-1})}\nonumber\\
&=\sum^{\infty}_{m,n=1}\frac{q^{mn2^a}(1-q^{2m-1})}{(1-q^{n 2^a})(1-q^{2m-1})}
-\sum^{\infty}_{m,n=1}\frac{q^{mn2^a+n2^{a-1}}-q^{mn2^a+2m-1}}{(1-q^{n 2^a})(1-q^{2m-1})}\nonumber\\
&=\sum^{\infty}_{m,n=1}\frac{q^{mn2^a}}{1-q^{n 2^a}}
-\sum^{\infty}_{m,n=1}\frac{q^{mn2^a+n2^{a-1}}-q^{mn2^a+2m-1}}{(1-q^{n 2^a})(1-q^{2m-1})}\nonumber\\
&=\sum^{\infty}_{n=1}\frac{q^{n 2^a}}{(1-q^{n 2^a})^2}
-\sum^{\infty}_{m,n=1}\frac{q^{mn2^a+n2^{a-1}}-q^{mn2^a+2m-1}}{(1-q^{n 2^a})(1-q^{2m-1})}\nonumber\\
&=\sum^{\infty}_{n=1}\sigma_1(n)q^{n 2^a}
-\sum^{\infty}_{m,n=1}\frac{q^{mn2^a+n2^{a-1}}-q^{mn2^a+2m-1}}{(1-q^{n 2^a})(1-q^{2m-1})}.
\end{align}
To prove the theorem, it suffices to prove
\begin{align}
[q^{n 2^{a}}]\sum^{\infty}_{m,n=1}\frac{q^{mn2^a+2m-1}-q^{mn2^a+n2^{a-1}}}{(1-q^{n 2^a})(1-q^{2m-1})}=0.
\end{align}
For convenience, we define
$$F(a):=\sum^{\infty}_{m,n=1}\frac{q^{mn2^a+2m-1}-q^{mn2^a+n2^{a-1}}}{(1-q^{n 2^a})(1-q^{2m-1})}.$$
Then by the similar argument of \eqref{1-5}, we can establish that
\begin{align}
F(a)=\sum^{\infty}_{m=1}\frac{q^{2m-1}}{1-q^{2m-1}}\sum^{\infty}_{n=m}\frac{q^{n 2^a} }{1-q^{n 2^a}}
-\sum^{\infty}_{m=1}\frac{1}{1-q^{2m-1}}\sum^{\infty}_{n=m+1}\frac{q^{n 2^a-2^{a-1}} }{1-q^{n 2^a-2^{a-1}}}. \label{3-1}
\end{align}
According to the values of $a$, we break our proof into two cases.

(1) Case $a=1$. Substituting $a=1$ into \eqref{3-1} and then combining the result with
Lemma \ref{lem-1} yields
\begin{align}
F(1)=\sum^{\infty}_{m,n=1}\frac{q^{2mn+2m-1}-q^{2 mn+n}}{(1-q^{2n})(1-q^{2m-1})}=0.
\end{align}
This completes the case of $a=1$.

(2) Case $a\geq2$.
Observe that
\begin{align}
F(a)&=\sum^{2^{a-1}}_{s=1}\sum^{\infty}_{m=1}\frac{q^{(2m-1)s}}{1-q^{2^{a} m-2^{a-1}}}\sum^{\infty}_{n=m}\frac{q^{n 2^a} }{1-q^{n 2^a}}\nonumber\\
&\qquad-\sum^{2^{a-1}-1}_{s=0}\sum^{\infty}_{m=1}\frac{q^{(2m-1)s}}{1-q^{m2^a-2^{a-1}}}\sum^{\infty}_{n=m+1}\frac{q^{n 2^a-2^{a-1}} }{1-q^{n 2^a-2^{a-1}}}\nonumber\\
&=\sum^{2^{a-1}-1}_{s=1}\sum^{\infty}_{m=1}\frac{q^{(2m-1)s}}{1-q^{ m 2^{a}-2^{a-1}}}\sum^{\infty}_{n=m}\frac{q^{n 2^a} }{1-q^{n 2^a}}\nonumber\\
&
-\sum^{2^{a-1}-1}_{s=1}\sum^{\infty}_{m=1}\frac{q^{(2m-1)s}}{1-q^{ m 2^{a}-2^{a-1}}}\sum^{\infty}_{n=m+1}\frac{q^{n 2^a-2^{a-1}} }{1-q^{n 2^a-2^{a-1}}}\nonumber\\
&+\sum^{\infty}_{m=1}\frac{q^{(2m-1)\cdot2^{a-1}}}{1-q^{ m 2^{a}-2^{a-1}}}\sum^{\infty}_{n=m}\frac{q^{n 2^a} }{1-q^{n 2^a}}
-\sum^{\infty}_{m=1}\frac{1}{1-q^{2^{a} m-2^{a-1}}}\sum^{\infty}_{n=m+1}\frac{q^{n 2^a-2^{a-1}} }{1-q^{n 2^a-2^{a-1}}}\nonumber\\
&=\sum^{2^{a-1}-1}_{s=1}\left(\sum^{\infty}_{m=1}\frac{q^{(2m-1)s}}{1-q^{ m 2^{a}-2^{a-1}}}\sum^{\infty}_{n=m}\frac{q^{ n 2^{a}} }{1-q^{n 2^a}}
-\sum^{\infty}_{m=1}\frac{q^{(2m-1)s}}{1-q^{ m 2^{a}-2^{a-1}}}\right.\nonumber\\
&\quad\left.\times\sum^{\infty}_{n=m+1}\frac{q^{ n 2^{a}-2^{a-1}} }{1-q^{ n 2^{a}-2^{a-1}}}\right)\nonumber\\
&=\sum^{2^{a-1}-1}_{s=1}\left(\sum^{\infty}_{m=1}\frac{q^{(2m-1)s}}{1-q^{ m 2^{a}-2^{a-1}}}\sum^{\infty}_{n=m}\frac{q^{n 2^a} }{1-q^{n 2^a}}
-\sum^{\infty}_{m=1}\frac{q^{m 2^a}}{1-q^{m 2^a}}
\sum^{\infty}_{n=m}\frac{q^{ n 2^{a-1}+s}}{1-q^{n 2^a+2s}}\right). \label{3-2}
\end{align}
Here we obtain the penultimate step by employing Lemma \ref{lem-1} with $q$ replaced by $q^{2^{a-1}}$ and the last step follows from Lemma \ref{lem-2} with $r=2^{a-1}$.

Notice that by direct calculation
\begin{align}
F(a)&=\sum^{2^{a-1}-1}_{s=1}q^{-s}\sum^{\infty}_{m=1}\sum^{\infty}_{k=0}q^{2ms+mk2^{a}-k 2^{a-1}}\sum^{\infty}_{n=m}\sum^{\infty}_{i=1}q^{ in 2^{a}}\nonumber\\
&\quad-\sum^{2^{a-1}-1}_{s=1}q^{s}\sum^{\infty}_{m=1}\sum^{\infty}_{k=1}q^{mk2^{a}}\sum^{\infty}_{n=m}q^{ n 2^{a-1}}\sum^{\infty}_{i=0}q^{ in 2^{a}+2is}\nonumber\\
&=\sum^{2^{a-1}-1}_{s=1}\sum^{\infty}_{m=1}\sum^{\infty}_{k=0}
\sum^{\infty}_{n=m}\sum^{\infty}_{i=1}q^{(2m-1)s+mk2^{a}- k 2^{a-1}+ in 2^{a}}\nonumber\\
&\quad-\sum^{2^{a-1}-1}_{s=1}\sum^{\infty}_{m=1}\sum^{\infty}_{k=1}\sum^{\infty}_{n=m}\sum^{\infty}_{i=0}
q^{(2i+1)s+mk2^{a}+ n 2^{a-1}+ in 2^{a}}.
\end{align}
Observe that
\begin{align*}
(2m-1)s+ mk 2^{a}- k 2^{a-1}+ in 2^{a}\equiv (2m-1)s \pmod{2^{a-1}}.
\end{align*}
Since $2m-1$ is odd and $1\leq s\leq 2^{a-1}-1$, it follows that
 \begin{align*}
(2m-1)s+ mk 2^{a}- k 2^{a-1}+ in 2^{a} \not\equiv 0 \pmod{2^{a-1}}.
\end{align*}
Then
\begin{align*}
(2m-1)s+ mk 2^{a}- k 2^{a-1}+ in 2^{a} \not\equiv 0 \pmod{2^{a}}.
\end{align*}
Similarly,
\begin{align*}
(2i+1)s+ mk 2^{a}+ n 2^{a-1}+ in 2^{a} \not\equiv 0 \pmod{2^{a}}.
\end{align*}
We thus obtain that for $a\geq2$,
$$[q^{n 2^a}]F(a)=0.$$
This completes the proof.

\section{Second Proof of Amdeberhan-Andrews-Ballantine Conjecture}\label{proof}

We first prove Theorem \ref{lem:base}.

\begin{proof}[Proof of Theorem \textup{\ref{lem:base}}]
Firstly we define two functions,
\[
D_j:=\frac{q^j}{1-q^{2j}},
\qquad
E_j:=\frac{q^{2j}}{1-q^{2j}}.
\]
Then,
\begin{equation}\label{eq:odd-even}
\frac{q^j}{1+q^j}=D_j-E_j.
\end{equation}
 Expanding the term $1/(1-q^{2m-1})$ on the left-hand side of \eqref{eq:base} and then interchanging the order of the summation, we get
\begin{align*}
\sum_{m,k\geq 1}
\frac{q^{2mk}}{(1+q^k)(1-q^{2m-1})}
&=
\sum_{k\geq 1}\frac{1}{1+q^k}
\sum_{r\geq 0}\sum_{m\geq 1}q^{2mk+r(2m-1)}  \\
&=
\sum_{k\geq 1}\frac{1}{1+q^k}
\sum_{r\geq 0}\frac{q^{2k+r}}{1-q^{2k+2r}}.
\end{align*}
Making the change of variable $h=k+r$ in the last inner sum of the above equation, we are led to
\begin{align}
\sum_{m,k\geq 1}
\frac{q^{2mk}}{(1+q^k)(1-q^{2m-1})}
&=\sum_{h\geq1}\sum_{1\leq k\leq h}
\frac{q^h}{1-q^{2h}}\frac{q^k}{1+q^k} \notag \\
&=\sum_{h\geq1}\sum_{1\leq k\leq h}D_h(D_k-E_k),\label{eq:F-start}
\end{align}
where \eqref{eq:odd-even} is invoked in the last step.

Now our main aim is to simplify the double sum on the right-hand side of \eqref{eq:F-start}. To that end, we split the first sum into two parts, separating the terms with $k=h$ from the remaining terms, to obtain
\begin{align}
\sum_{m,k\geq 1}\frac{q^{2mk}}{(1+q^k)(1-q^{2m-1})}
&=\sum_{h\geq 1}D_h^2+\sum_{h\geq1}\left(\sum_{1\leq k<h}D_kD_h-\sum_{1\leq k\leq h}E_kD_h
\right).\label{eq:F-split}
\end{align}
Observe that it suffices to show that
\begin{align}\label{zero}
	\sum_{h\geq1}\left(\sum_{1\leq k<h}D_kD_h-\sum_{1\leq k\leq h}E_kD_h
	\right)=0,
\end{align}
for then the required result follows.

To achieve the above aim,  consider the following identity, for $r,s\geq 1$,
\begin{equation}\label{eq:telescoping-block}
D_rD_{r+s}=E_rD_s-D_sE_{r+s}.
\end{equation}
This can be seen easily by substituting the values of $D_j$ and $E_j$, that is
\[
\begin{aligned}
	D_rD_{r+s}
	&=\frac{q^r}{1-q^{2r}}\cdot \frac{q^{r+s}}{1-q^{2r+2s}}
	=\frac{q^{2r+s}}{(1-q^{2r})(1-q^{2r+2s})} \\
	&= \frac{q^{2r+s}}{(1-q^{2r})(1-q^{2r+2s})}
	=E_rD_s-D_sE_{r+s}.
\end{aligned}
\]
Summing \eqref{eq:telescoping-block} over $r,s\geq 1$ leads us to
\begin{align}
\sum_{r,s\ge1} D_rD_{r+s}&= \sum_{r,s\ge1}\bigl(E_rD_s-D_sE_{r+s}\bigr)\nonumber\\
&=\sum_{s\geq1}D_s\left(\sum_{r\geq1}\left(E_r-E_{r+s}\right)\right)\nonumber\\
&=\sum_{s\geq1}D_s\sum_{r=1}^sE_r.
\end{align}
Making the change of variables $r=k$ and $h=r+s$ on the left-hand side, we arrive at
\begin{align*}
\sum_{h\ge 1}\sum_{k=1}^{h-1} D_kD_h=\sum_{s\geq1}\sum_{1\leq r\leq s}D_sE_r.
\end{align*}
This proves our aim in \eqref{zero}.

Therefore equations \eqref{eq:F-split} and \eqref{zero} together imply that
\[
\sum_{m,k\geq 1}\frac{q^{2mk}}{(1+q^k)(1-q^{2m-1})} =\sum_{h\geq 1}D_h^2=\sum_{h\geq 1}\frac{q^{2h}}{(1-q^{2h})^2}.
\]
This completes the proof of \eqref{eq:base}.
\end{proof}

We now provide the second proof of Conjecture \ref{aab conjecture 1}.
\begin{proof}[Proof of Theorem \textup{\ref{thm:main}}]
Let
\[
\delta:=2^{a-1},
\qquad
x:=q^\delta.
\]
The series in \eqref{eq:main} then becomes
\begin{align}\label{series}
\sum_{m,k\geq 1}
\frac{x^{2mk}}{(1+x^k)(1-q^{2m-1})}.
\end{align}
Expanding
\[
\frac{1}{1-q^{2m-1}}=\sum_{\ell\geq 0}q^{\ell(2m-1)},
\]
and substituting in \eqref{series}, we have
\begin{align*}
\sum_{m,k\geq 1}
\frac{x^{2mk}}{(1+x^k)(1-q^{2m-1})}=	\sum_{m,k\geq 1}\frac{x^{2mk}}{(1+x^k)}\sum_{\ell\geq 0}q^{\ell(2m-1)}.
\end{align*}
We observe that all factors  except $q^{\ell(2m-1)}$ are powers of $x=q^\delta$.  What we want is the coefficient of $q^{2N\delta}$.

To extract the coefficient of
\begin{align*}
	[q^{N2^a}]=[q^{2N\delta}],
\end{align*}
we consider
\begin{align*}
	q^{mk2^a+\ell(2m-1)}=q^{2mk\delta +\ell(2m-1)}.
\end{align*}
That means, we must have
\begin{align*}
	 &2mk\delta +\ell(2m-1)=2N\delta
	 \implies&\ell(2m-1)=\delta(2N-2mk).
\end{align*}
Since $2m-1$ is odd and $\delta$ is a power of $2$, the congruence
\[
\ell(2m-1)\equiv 0\pmod \delta
\]
holds if and only if $\ell\equiv 0\pmod \delta$.  Thus only the terms with $\ell=\delta t$ contribute to the coefficient of $q^{2N\delta}$ in the sum.  For such terms,
\[
q^{\ell(2m-1)}=q^{\delta t(2m-1)}=x^{t(2m-1)}.
\]
Consequently, for the purpose of extracting the coefficient $[x^{2N}]$, the series reduces to
\begin{align*}
	\sum_{m,k\geq 1}
	\frac{x^{2mk}}{(1+x^k)(1-x^{2m-1})}.
\end{align*}
Therefore, we finally have
\begin{align}\label{eq:reduction-to-x}
&\coeff{q^{N2^a}}
\sum_{m,k\geq 1}
\frac{q^{mk2^a}}{(1+q^{k2^{a-1}})(1-q^{2m-1})}
\notag\\
&\hspace{2cm}=
\coeff{x^{2N}}
\sum_{m,k\geq 1}
\frac{x^{2mk}}{(1+x^k)(1-x^{2m-1})}.
\end{align}
By Theorem \ref{lem:base}, we already have
\begin{align*}
\coeff{x^{2N}}\sum_{h\geq 1}\frac{x^{2h}}{(1-x^{2h})^2}=\sigma_1(N).
\end{align*}
This proves \eqref{eq:main}.
\end{proof}

\section{Concluding Remarks}\label{concluding}

Our proof of Conjecture \ref{aab conjecture 1} leads us to a new representation of the quasi-modular form $E_2(q)$, given in Corollary \ref{new representation}.   This naturally raises the following question:

\emph{Is there a generalization of Theorem \ref{lem:base} for the generalized divisor function} $$\sigma_k(n)=\sum_{0<d|n}d^k\ ?$$

If so, it would lead to a new representation for the Eisenstein series $E_{k+1}(q),\ k\in\mathbb{N}$.

Another question is the following:

\emph{Does $\sigma_k(n)$ also arise as coefficients of a double Lambert series similar to the one for $\sigma_1(n)$ in Theorem \textup{\ref{thm:main}}?}

\medskip

\noindent
  {\bf{Acknowledgements:}}
We are grateful to Professor Andrews for reading an early version of the first proof in this paper.  We also thank H.H. Chan and S. Chern for their valuable comments and suggestions.
The first author was supported by the Taishan Scholar Project of Shandong Province (No. tsqn202507173) and the Natural Science Foundation of Shandong Province of China (No. ZR2025MS91). The second author was partially supported by the Grant ANRF/ECRG/2024/003222/PMS of the Anusandhan National Research Foundation (ANRF), Govt. of India, and CPDA and FIG grants of IIT Roorkee.  The second author was supported through the aforementioned ANRF grant as an ANRF project student. Authors thank these institutions for the support.

\end{document}